\documentclass[11pt,reqno]{amsart}
\usepackage{latexsym,amsmath}
\usepackage{amsfonts}
\usepackage{amssymb}
\usepackage{fixmath}

\usepackage{fullpage}
\usepackage{graphicx}
\usepackage[colorlinks=true,citecolor=blue]{hyperref}
\usepackage[dvipsnames]{xcolor}
\usepackage{enumerate}

\usepackage[T1]{fontenc}
\usepackage{fourier}
\usepackage{bbm}

\usepackage[boxed, linesnumbered, noend, noline]{algorithm2e}
\SetKwInput{KwData}{Input}
\SetKwInput{KwResult}{Output}

\numberwithin{equation}{section}

\newcommand\ra{\to}
\renewcommand{\epsilon}{\eps}

\newtheorem{definition}{Definition}[section]

\newtheorem{remark}[definition]{Remark}
\newtheorem{theorem}[definition]{Theorem}
\newtheorem{lemma}[definition]{Lemma}
\newtheorem{proposition}[definition]{Proposition}

\newcommand\eps{\varepsilon}

\newcommand\GG{\mathbb{G}}

\newcommand\Var{\mathrm{Var}}

\def\?#1{}
\makeatletter
\def\whp{w.h.p\@ifnextchar-{.}{\@ifnextchar.{.\?}{\@ifnextchar,{.}{\@ifnextchar){.}{\@ifnextchar:{.:\?}{.\ }}}}}}
\makeatother

\makeatletter
\def\Whp{W.h.p\@ifnextchar-{.}{\@ifnextchar.{.\?}{\@ifnextchar,{.}{\@ifnextchar){.}{\@ifnextchar:{.:\?}{.\ }}}}}}
\makeatother

\newcommand{\Erdos}{Erd\H{o}s}
\newcommand{\Renyi}{R\'enyi}

\newcommand{\Bollobas}{Bollob\'as}

\newcommand\pr{\mathbb{P}}
\renewcommand\Pr{\pr}

\newcommand\Lem{Lemma}
\newcommand\Prop{Proposition}
\newcommand\Thm{Theorem}

\def\pr{{\mathbb P}}

\begin{document}

	\title{Asymptotic normality of the giant component size in a random bipartite graph}
	\author{Dmitry Shabanov}
	\address{Dmitry Shabanov, Laboratory of Combinatorial and Geometric Structures, Moscow Institute of Physics and Technology, Russia}
	\author{Pavel Zakharov}
	\address{Pavel Zakharov, Department of Computer Science, TU Dortmund University, Germany}

	\begin{abstract}%
		This paper studies the giant component of the sparse random bipartite graph $\GG(n, n, p)$, where $p = c/n$ for a fixed constant $c > 1$. We prove that its size is asymptotically normal.
	\end{abstract}

	\maketitle

	\section{Introduction}

	\subsection{Main definitions}

	One of the classical random graph models is an \Erdos-\Renyi\ random graph $\GG(n,p)$.
	In this model each pair of vertices from an $n$-element set is joined by an edge independently with probability $p$.
	In other words, we perform a Bernoulli scheme on the edges of the complete graph $K_n$ on $n$ vertices.

	The model $\GG(n,p)$ is a central object in the theory of random graphs, but not the only one.
	Generalisations where the Bernoulli scheme is performed on the edges of a different base graph have been studied extensively.
	In this work the base graph is the complete bipartite graph $K_{n,n}$: it consists of two disjoint sets of $n$ vertices each (the two {parts}), all possible edges between the parts are present, and there are no edges inside the parts.
	By a random graph $\GG(n,n,p)$ we denote a random element, resulting from deleting each edge from $K_{n,n}$ independently with probability $1 - p$.

	\subsection{Known results}

	One of the most fundamental results in the theory of random graphs is the discovery of the ``double jump'' phenomenon by \Erdos\ and \Renyi\ \cite{ErdosRenyi}.
	Let $C_i(n) = C_i(n,p)$ be the size of the $i$th largest component of $\GG(n,p)$.
	\begin{theorem}[\cite{ErdosRenyi}]\label{thm:1}
		Let $p = c/n$, where $c > 0$ is fixed and does not depend on $n$.
		\begin{enumerate}
			\item If $c < 1$, then
			\[
			\frac{C_1(n)}{\ln n} \ra \alpha(c) = \frac{1}{c - \ln c - 1}, \quad n \to \infty
			\]
			in probability.

			\item If $c = 1$, then $C_1(n)$ is of order $n^{2/3}$ in probability:
			\[
			C_1(n) = \Theta_{\Pr}(n^{2/3}).
			\]
			In other words, for each growing function $w(n) \ra +\infty$,
			\[
			\Pr\left(n^{2/3}/w(n) \le C_1(n) \le n^{2/3} w(n)\right) \to 1 \text{ as } n \to \infty.
			\]

			\item If $c > 1$,
			\[
			\frac{C_1(n)}{n} \ra \beta(c), \quad n \to \infty
			\]
			in probability, where $\beta(c)$ is the unique solution of $\beta + e^{-\beta c} = 1$ on the interval $(0,1)$.
			In addition,
			\[
			\frac{C_2(n)}{\ln n} \ra \alpha(c) = \frac{1}{c - \ln c - 1}, \quad n \to \infty
			\]
			in probability.

		\end{enumerate}
	\end{theorem}
	Thus when $p$ crosses $1/n$, the order of magnitude of $C_1(n)$ first changes from $\ln n$ to $n^{2/3}$, and then from $n^{2/3}$ to $n$.
	Moreover, in the latter case the component of linear size is unique, and all others are at most logarithmic in size.
	Therefore it makes sense to call the largest component in this regime ``giant''.

	The results of \Thm~\ref{thm:1} have been strengthened over the years.
	In particular, Stepanov established the asymptotic normality of the size of the giant component.
	\begin{theorem}[\cite{Stepanov}]\label{thm:2}
		Let $np \to c$, where $c > 1$ is fixed.
		Then
		\[
		\sqrt{n}\left(\frac{C_1(n)}{n} - \beta\right) \ra \mathcal{N}(0, \sigma^2), \quad n \to \infty
		\]
		in distribution, where $\beta = \beta(c)$ is the unique solution of $\beta + e^{-\beta c} = 1$ on the interval $(0,1)$, and
		\[
		\sigma^2 = \frac{\beta(1-\beta)}{(1-c(1-\beta))^2}.
		\]
	\end{theorem}
	The original proof of Stepanov is rather involved and relies on the exact asymptotics of the probability that a random graph is connected.
	In 2012 \Bollobas\ and Riordan presented a shorter proof which is based on a martingale central limit theorem \cite{BollobasRiordanWalk}.

	The question of the limiting distribution of $C_1(n)$ in the critical case $c = 1$ remained open for a long time, until it was resolved by Aldous \cite{Aldous} in 1997.

	A similar phase transition in the component structure occurs in other models as well: in random hypergraphs~\cite{SchmidtPruzanShamir}, random subgraphs of the hypercube \cite{AjtaiKomlosSzemeredi}, and random subgraphs of pseudo-random graphs \cite{FriezeKrivelevichMartin}.
	However, these results only concern the first-order asymptotics (the ``law of large numbers'').
	Obtaining the precise limiting distribution turned out to be much more challenging; however, an analogue of \Thm~\ref{thm:2} was established for random hypergraphs \cite{BehrischCojaOghlanKang, BollobasRiordanHypergraph}.

	Component sizes of the random bipartite graph $\GG(n,n,p)$ have also been studied.
	In particular, Johansson proved an analogue of \Thm~\ref{thm:1} for this model.
	Denote by $L_i(n)$ the size of the $i$th largest component of $\GG(n,n,p)$.

	\begin{theorem}[\cite{Johansson}]\label{thm:3}
		Let $np \to c$ as $n \to \infty$, where $c > 0$ is fixed.
		\begin{enumerate}
			\item If $c < 1$, then there exists $A = A(c)$, such that
			\[
			\Pr(L_1(n) \le A \cdot \ln n) \to 1, \quad n \to \infty.
			\]

			\item If $c > 1$, then
			\[
			\frac{L_1(n)}{n} \ra 2\beta(c), \quad n \to \infty
			\]
			in probability, where $\beta(c)$ is the unique solution of $\beta + e^{-\beta c} = 1$ on the interval $(0,1)$.
			In addition there exists $B = B(c)$, such that
			\[
			\Pr(L_2(n) \le B \cdot \ln n) \to 1, \quad n \to \infty.
			\]
		\end{enumerate}
	\end{theorem}
	Component sizes of $\GG(n,n,p)$ were also studied when $np \ra 1$ \cite{DoErdeKang, Wang}.

	\subsection{New result}

	The main result of the present work strengthens \Thm~\ref{thm:3} and establishes the asymptotic normality of the size $L_1(n)$ of the giant component of $\GG(n,n,p)$ above the phase transition.

	\begin{theorem}\label{thm:4}
		Let $np \to c$ as $n \to \infty$, where $c > 1$ is fixed. Then
		\begin{equation}\label{eq:main}
			\sqrt{n}\left(\frac{L_1(n)}{n} - 2\beta\right) \rightarrow \mathcal{N}(0, 2\sigma^2), \quad n \to \infty
		\end{equation}
		in distribution,
		where $\beta = \beta(c)$ is the unique solution of $\beta + e^{-\beta c} = 1$ on the interval $(0,1)$, and
		\[
		\sigma^2 = \frac{\beta(1-\beta)}{(1-c(1-\beta))^2}.
		\]
	\end{theorem}

	\begin{remark}
		This result also follows from the recent, more general work of Clancy~\cite{Clancy}. However, the Russian-language version of the present paper~\cite{ZakharovShabanov} was published earlier.
	\end{remark}

	\section{Proof of \Thm~\ref{thm:4}}
	We are going to prove the following lemma, which implies \Thm~\ref{thm:4}.

	\begin{lemma}\label{lem:1}
		Let $\GG^{(1)}(n,p)$, $\GG^{(2)}(n,p)$ be independent \Erdos-\Renyi\ random graphs, and let $C^{(1)}(n)$, $C^{(2)}(n)$ be the sizes of their largest components.
		Then there exists a random variable $\xi_n$ such that for all $n$,
		\begin{equation}
			L_1(n) + \xi_n \overset{d}{=} C^{(1)}(n) + C^{(2)}(n),
		\end{equation}
		and $\xi_n/\sqrt{n} \ra 0$ in probability as $n \to \infty$.
	\end{lemma}

	Indeed, if \Lem~\ref{lem:1} holds, then \Thm~\ref{thm:2} implies that
	\[
	\sqrt{n}\left(\frac{C^{(1)}(n) + C^{(2)}(n)}{n} - 2\beta\right) \ra \mathcal{N}(0, 2\sigma^2)
	\]
	in distribution,
	as  $\GG^{(1)}(n,p)$ and $\GG^{(2)}(n,p)$ are independent.
	On the other hand, according to \Lem~\ref{lem:1}, the following sequence has exactly the same limit in distribution:
	\begin{equation}\label{eq:1}
		\sqrt{n}\left(\frac{L_1(n) + \xi_n}{n} - 2\beta\right) \overset{d}{=} \sqrt{n}\left(\frac{L_1(n)}{n} - 2\beta\right) + \frac{\xi_n}{\sqrt{n}}.
	\end{equation}
	Now \eqref{eq:1} together with $\xi_n/\sqrt{n} \ra 0$ establishes \eqref{eq:main}.

	Next we sketch the proof of \Lem~\ref{lem:1} using the component exploration process.

	\subsection{Component exploration process}
	For a random graph $\GG(n,p)$ with vertex set $V_n$ the following exploration process (see, e.g.,~\cite{BollobasRiordanHypergraph}) is used. At any moment every vertex is in one of three states: \emph{explored}, \emph{active}, or \emph{unseen}.
	\begin{enumerate}
		\item At each step $t=0,1,2,\dots$ there are three sets $S_t, A_t, U_t$ and a number $C_t$, where $S_t$ is the set of explored vertices, $A_t$ is the set of active vertices, $U_t$ is the set of unseen vertices, and $C_t$ is the number of components explored by time $t$. Initially we choose an arbitrary vertex $v$: $A_0 = \{v\}$, $S_0 = \emptyset$, $U_0 = V_n \setminus \{v\}$, $C_0 = 0$.

		\item Suppose $(S_t, A_t, U_t, C_t)$ have been constructed. If $|A_t| > 0$, then at step $t+1$ the first active vertex $v_t \in A_t$ is chosen, and we let $X_{t+1} = N(v_t) \cap U_t$, where $N(v)$ denotes the set of neighbours of a vertex $v$.
		Then
		\[
		S_{t+1} = S_t \cup \{v_t\}, \quad A_{t+1} = X_{t+1} \cup A_t \setminus \{v_t\}, \quad U_{t+1} = U_t \setminus X_{t+1}, \quad C_{t+1} = C_t.
		\]

		\item If $|A_t| = 0$ and $|U_t| > 0$, then the current component is fully explored. In this case at step $t+1$ the first unseen vertex $v_t \in U_t$ is chosen:
		\[
		S_{t+1} = S_t, \quad A_{t+1} = \{v_t\}, \quad U_{t+1} = U_t \setminus \{v_t\}, \quad C_{t+1} = C_t + 1.
		\]

		\item If $|A_t| = 0$ and $|U_t| = 0$, the process terminates.
	\end{enumerate}

	Note that the random variable  $|X_{t+1}|$, which is the number of newly added active vertices, has the conditional binomial distribution $\mathrm{Bin}(|U_t|, p)$, provided there is a current active vertex.

	For a bipartite random graph $\GG(n,n,p)$ we modify the process, such that it goes in parallel in both parts. By $V_n^{(1)}$ and $V_n^{(2)}$ denote the two parts of  $\GG(n,n,p)$.

	\begin{enumerate}
		\item At each step there are two quadruples, one for each part, $(S_t^{(i)}, A_t^{(i)}, U_t^{(i)}, C_t)$, $i=1,2$. We initialize them as follows: arbitrary $v \in V_n^{(1)}$, $A_0^{(1)} = \{v\}$, $A_0^{(2)} = \emptyset$, $S_0^{(i)} = \emptyset$, $U_0^{(1)} = V_n^{(1)} \setminus \{v\}$, $U_0^{(2)} = V_n^{(2)}$, $C_0 = 0$.

		\item If $|A_t^{(i)}| > 0$ for $i=1,2$, then at step $t+1$ we take $v_t^{(i)}$, the first active vertex of $A_t^{(i)}$, $i = 1,2$, and let $X_{t+1}^{(i)} = N(v_t^{(i)}) \cap U_t^{(3-i)}$.
		\[
		S_{t+1}^{(i)} = S_t^{(i)} \cup \{v_t^{(i)}\}, \quad A_{t+1}^{(i)} = X_{t+1}^{(3-i)} \cup A_t^{(i)} \setminus \{v_t^{(i)}\}, \quad U_{t+1}^{(i)} = U_t^{(i)} \setminus X_{t+1}^{(3-i)}, \quad C_{t+1} = C_t.
		\]

		\item If without loss of generality $|A_t^{(1)}| > 0$ but $|A_t^{(2)}| = 0$, then the component is not fully explored yet, and only neighbours of the first active vertex from $V_n^{(1)}$ will be added.
		Specifically, let $v_t^{(1)}$ be the first active vertex from $A_t^{(1)}$, and $X_{t+1}^{(1)} = N(v_t^{(1)}) \cap U_t^{(2)}$:
		\[
		S_{t+1}^{(1)} = S_t^{(1)} \cup \{v_t^{(1)}\}, \quad A_{t+1}^{(1)} = A_t^{(1)} \setminus \{v_t^{(1)}\}, \quad U_{t+1}^{(2)} = U_t^{(2)} \setminus X_{t+1}^{(1)}, \quad A_{t+1}^{(2)} = X_{t+1}^{(1)}.
		\]
		All the other elements of both quadruples remain unchanged. The case $|A_t^{(1)}|=0$, $|A_t^{(2)}|>0$ is symmetric.

		\item If $|A_t^{(i)}| = 0$ for $i=1,2$ and $|U_t^{(1)} \cup U_t^{(2)}| > 0$, then the current component is fully explored. At step $t+1$ we pick the first unseen vertex $v_t^{(i)} \in U_t^{(i)}$, where $U_t^{(i)}$ is nonempty, and
		\[
		A_{t+1}^{(i)} = \{v_t^{(i)}\}, \quad U_{t+1}^{(i)} = U_t^{(i)} \setminus \{v_t^{(i)}\}, \quad C_{t+1} = C_t + 1.
		\]
		The other quantities remain unchanged.

		\item If $|A_t^{(1)} \cup A_t^{(2)}| = 0$ and $|U_t^{(1)} \cup U_t^{(2)}| = 0$, then the exploration process terminates.
	\end{enumerate}

	Now we describe the proof strategy for \Lem~\ref{lem:1}. Let $T_0$ be the time at which the described process starts exploring the giant component of $\GG(n,n,p)$. First we show that this happens fast enough.

	\begin{proposition}\label{prop:1}
		\[
		\Pr\left(T_0 \le (\ln n)^2\right) \to 1, \quad n \to \infty.
		\]
	\end{proposition}
	For the sake of the analysis we are interested in a moment of time at which both $A_t^{(1)}$ and $A_t^{(2)}$ are large enough. Let $t_0 = \lceil (1 + (c-1)/8) n^{1/4} \rceil$; the next proposition shows that this choice works.

	\begin{proposition}\label{prop:2}
		There exists such $\alpha = \alpha(c) > 0$, that
		\[
		\Pr\left(|A_{t_0}^{(1)}| \ge \alpha n^{1/4}, \; |A_{t_0}^{(2)}| \ge \alpha n^{1/4}\right) \to 1, \quad n \to \infty.
		\]
	\end{proposition}

	Further let $\tau_n = \min\{t : t \ge t_0, \; |A_t^{(1)}| = 0 \text{ or } |A_t^{(2)}| = 0\}$ be the first moment of time after $t_0$ at which one of the two parts has no active vertices left.
	Note that for any $t \in (t_0, \tau_n)$
	\[
	|A_t^{(i)}| = |A_{t-1}^{(i)}| - 1 + |X_t^{(3-i)}|, \quad i = 1,2.
	\]

	Next, the random variable $|X_t^{(3-i)}|$ has the conditional binomial distribution $\mathrm{Bin}(|U_t^{(i)}|, p)$, where
	\[
	|U_t^{(i)}| = n - \sum_{j=t_0}^{t} |X_j^{(3-i)}| - |A_{t_0}^{(i)} \cup S_{t_0}^{(i)}|.
	\]
	Here $A_{t_0}^{(i)} \cup S_{t_0}^{(i)}$ is the set of either active or explored vertices at time $t_0$ in the $i$th part.
	Clearly, $|S_{t_0}^{(i)}| \le t_0 = o(\sqrt{n})$, and the size of
	$A_{t_0}^{(i)}$  can be bounded using Markov's inequality.
	\begin{proposition}\label{prop:3}
		For every $i = 1,2$
		\[
		\Pr\left(|A_{t_0}^{(i)}| \le n^{1/3}\right) \to 1, \quad n \to \infty.
		\]
	\end{proposition}
	Thus for $t \in (t_0, \tau_n)$ our process looks like a parallel exploration of two random graphs $\GG(n,p)$.
	Moreover, with $(S_{t_0}^{(i)}, A_{t_0}^{(i)}, U_{t_0}^{(i)}, C_{t_0})$ fixed for $i = 1,2$, we have two independent processes, i.e.\ in effect we explore the giant components of two independent random graphs $\GG^{(1)}(n,p)$, $\GG^{(2)}(n,p)$.
	Since at time $t_0$ we already have many active vertices in both parts (more than the total number of vertices in non-giant components, according to \Thm~\ref{thm:1}), we conclude that $\tau_n$ coincides in distribution with $\min(C^{(1)}(n), C^{(2)}(n))$ up to an error of order $o_{\Pr}(\sqrt{n})$.

	Now let us understand how  $L_1(n)$ and $\tau_n$ are connected.
	Clearly all the vertices considered between times $T_0$ and $\tau_n$ will go into the giant component, and there are between $2\tau_n - 2t_0$ and $2\tau_n$ of them.
	However, it is possible that in one part there are still active vertices left.
	Denote their number by $Z_n$, and let $\sigma_n$ be the number of vertices that will be added to the giant component after step $\tau_n$.
	The next proposition explains the connection between these random variables.
	\begin{proposition}\label{prop:4}
		\[
		\frac{\sigma_n - Z_n \frac{1}{1-c(1-\beta)}}{\sqrt{n}} \ra 0
		\]
		in probability as $n \ra +\infty$.
	\end{proposition}
	It remains to study the joint distribution of $(Z_n, \tau_n)$. Denote by $Y_n$ the number of active vertices remaining in one of the random graphs $\GG^{(i)}(n,p)$ at the moment when the exploration of the giant component in the other one is complete.
	\begin{proposition}\label{prop:5}
		There exist random variables $\eta_n'$, $\eta_n''$ such that for all $n$
		\[
		(\tau_n + \eta_n', \; Z_n + \eta_n'') \overset{d}{=} \left(\min(C^{(1)}(n), C^{(2)}(n)), \; Y_n\right).
		\]
		Moreover, $(|\eta_n'| + |\eta_n''|)/\sqrt{n} \ra 0$ in probability as $n \ra +\infty$.
	\end{proposition}

	Finally, we are ready to derive \Lem~\ref{lem:1}.

	\subsection{Proof of \Lem~\ref{lem:1}}
	First observe that \Prop~\ref{prop:3} implies that with high probability
	\[
	|L_1(n) - \sigma_n - 2\tau_n| \le 2t_0 = O(n^{1/4}),
	\]
	as the left-hand side does not exceed the total number of vertices considered by time $t_0$.
	Using \Prop~\ref{prop:4} we obtain
	\[
	\left|L_1(n) - Z_n \cdot \frac{1}{1-c(1-\beta)} - 2\tau_n\right| = o_{\Pr}(\sqrt{n}).
	\]
	Next, \Prop~\ref{prop:5} implies that, up to an error of order $o_{\Pr}(\sqrt{n})$, the random variable $L_1(n)$ coincides in distribution with
	\[
	Y_n \cdot \frac{1}{1-c(1-\beta)} + 2\min(C^{(1)}(n), C^{(2)}(n)).
	\]
	Finally, note that from the proof of \Prop~\ref{prop:4} (see below) it follows that
	\[
	\left| Y_n \cdot \frac{1}{1-c(1-\beta)} - |C^{(1)}(n) - C^{(2)}(n)| \right| = o_{\Pr}(\sqrt{n}),
	\]

	and since
	\[
	|C^{(1)}(n) - C^{(2)}(n)| + 2\min(C^{(1)}(n), C^{(2)}(n)) = C_1(n) + C_2(n),
	\]
	the proof of \Lem~\ref{lem:1} is complete.

	\subsection{Proofs of Propositions~\ref{prop:1}--\ref{prop:5}}

	\begin{proof}[Proof of \Prop~\ref{prop:1}]
		The proof of \Thm~\ref{thm:3} shows that each vertex belongs to the giant component with probability $\beta(1+o(1))$.
		Assume that by time $t_1 = \lceil (\ln n)^2 \rceil$ we have not started exploring the giant component yet.
		Then \Thm~\ref{thm:3} implies that with high probability we have only explored components of size at most $B \cdot \ln n$, and hence by time $t_1$ we have fully explored at least $B^{-1} \cdot \ln n$ components.

		Each time we start a new component, there are at least $n(1 + o(1))$ unseen vertices left in each part, so the conditional probability of starting the giant component equals $\beta(1+o(1))$ regardless of the past. Hence the probability that none of the first $B^{-1} \cdot \ln n$ components is the giant one is at most
		\[
		(1 - \beta(1+o(1)))^{B^{-1}\cdot \ln n} \to 0, \quad n \to \infty.
		\]
	\end{proof}

	\begin{proof}[Proof of \Prop~\ref{prop:2}]
		Set $t_1 = \lceil (\ln n)^2 \rceil$, $t_2 = \lceil n^{1/4} \rceil$.
		According to \Prop~\ref{prop:1} it is sufficient to consider the event that the exploration of the giant component starts before time $t_1$.
		First, let us provide a crude upper bound on the number of active vertices at time $t_1$.
		Clearly, the number of active vertices added at each step is stochastically dominated by  $\mathrm{Bin}(n,p)$.
		Thus
		\[
		\mathbb{E}|A_{t_1}^{(i)}| \le np\cdot t_1 = O(t_1).
		\]
		Hence $\Pr(|A_{t_1}^{(1)} \cup A_{t_1}^{(2)}| \ge \lfloor n^{1/5} \rfloor) \to 0$ as $n \to \infty$.

		Instead of the parallel procedure, consider the exploration process in which at each step we process a single active vertex in only one of the parts.
		Let us estimate the probability that after $t \in [t_2, 2t_2]$ such steps the total number of active vertices is small.
		Specifically, we bound the probability that it is at most $(c-1)t/2$.
		Note that the probability of such event does not exceed the probability that
		\[
		\left\{ \sum_{j=t_1}^{t} W_j \le \frac{c-1}{2} t + t \right\},
		\]
		where $W_j$ are independent binomial random variables $\mathrm{Bin}\left(n - t_1 - \lfloor n^{1/5} \rfloor - t - \left\lceil \frac{c-1}{2} t \right\rceil, \, p\right)$, as the number of new neighbours of the current active vertex stochastically dominates such a $W_j$, and we remove one active vertex at each step.
		As a consequence, we have the following upper bound:
		\[
		\Pr\left( \sum_{j=t_1}^{t} W_j \le \frac{c+1}{2} t \right) = \Pr\left( \sum_{j=t_1}^{t} (W_j - \mathbb{E}W_j) \le \frac{1-c}{2} t (1+o(1)) \right) \le
		\]
		(Chernoff bound)
		\[
		\le \exp\left\{ -\frac{(1-c)^2 t^2 (1+o(1))}{8ct} \right\} = \exp\left\{ -\frac{(1-c)^2}{8c} t (1+o(1)) \right\} \le
		\]
		\[
		\le \exp\left\{ -\frac{(1-c)^2}{8c} n^{1/4} (1+o(1)) \right\} = o\left(\frac{1}{n}\right).
		\]
		Thus with high probability at any time $t \in [t_2, 2t_2]$ we will have at least $(c-1)t/2$ active vertices in the ``slow'' exploration process.
		For the parallel exploration process this means that at time $t_2$ with high probability $|A_{t_2}^{(1)} \cup A_{t_2}^{(2)}| \ge (c-1)t_2/2$.
		Without loss of generality assume that $|A_{t_2}^{(1)}| \ge (c-1)t_2/4$.
		Therefore, the same argument as above implies that after additional $(c-1)t_2/8$ steps of the exploration process there will be at least $(c-1)t_2/8$ active vertices left in the first part, and at least $(c-1)^2 t_2/16$ active vertices in the second one.
	\end{proof}

	\begin{proof}[Proof of \Prop~\ref{prop:3}]
		First we estimate the expected value of $|A_{t_0}^{(i)}|$.
		Note that the number of added active vertices at each step is stochastically dominated by $\mathrm{Bin}(n,p)$.
		Therefore,
		\[
		\mathbb{E}|A_{t_0}^{(i)}| \le np \cdot t_0 = O(n^{1/4}).
		\]
		Hence Markov's inequality implies that $\Pr(|A_{t_0}^{(i)}| \ge n^{1/3}) \to 0$ as $n \to \infty$.
	\end{proof}

	\begin{proof}[Proof of \Prop~\ref{prop:4}]
		First let us understand how the situation looks at time $\tau_n$, when one of the parts has no active vertices anymore, while the other one still does.
		There are $Z_n$ active vertices left, and without loss of generality we assume that they belong to $V_n^{(1)}$.
		\Prop~\ref{prop:5} implies that $Z_n = O_{\Pr}(\sqrt{n})$.
		Hence in each part there are $(1-\beta)n(1+o(1))$ unexplored vertices left.
		Since $c(1-\beta) < 1$, we are in the situation, where from each of $Z_n$ active vertices in $V_n^{(1)}$ we start a component exploration process in a subcritical bipartite random graph.

		According to \Thm~\ref{thm:3}, each of these components will be at most logarithmic in size, so at every step each active vertex has $(1-\beta)n(1+o(1))$ unseen vertices in the opposite part as potential neighbours
		Thus $\sigma_n$ is the sum of the sizes of the explored components in a random bipartite graph $\GG((1-\beta)n(1+o(1)), (1-\beta)n(1+o(1)), p)$.
		We write
		\[
		\sigma_n = \sum_{j=1}^{Z_n} \xi_j(n),
		\]
		where $\xi_j(n)$ is the size of the component explored after time $\tau_n$ starting from the $j$th active vertex.
		Therefore, uniformly over all $j \le n^{2/3}$,
		\[
		\lim_{n \to \infty} \mathbb{E}\xi_j(n) = \frac{1}{1-c(1-\beta)}.
		\]
		Next we use the independence of the $\xi_j(n)$ and $Z_n$, as well as Chebyshev's inequality:
		\[
		\Pr\left( \left| \sigma_n - \sum_{j=1}^{Z_n} \mathbb{E}\xi_j(n) \right| > \varepsilon \sqrt{n} \;\middle|\; Z_n = x \right) \le \frac{\sum_{j=1}^{x} \Var\xi_j(n)}{\varepsilon^2 n} = O\left(\frac{x}{n}\right)
		\]
		uniformly over all $x \le n^{2/3}$.
		Hence
		\[
		\frac{\sigma_n - \sum_{j=1}^{Z_n} \mathbb{E}\xi_j(n)}{\sqrt{n}} \ra 0
		\]
		in probability.
		Finally, note that since $Z_n = O_{\Pr}(\sqrt{n})$,
		\[
		\frac{Z_n \cdot \frac{1}{1-c(1-\beta)} - \sum_{j=1}^{Z_n} \mathbb{E}\xi_j(n)}{\sqrt{n}} \ra 0
		\]
		in probability.
	\end{proof}

	\begin{proof}[Proof of \Prop~\ref{prop:5}]
		Starting from time $t_0$ the parallel component exploration will be identical to the simultaneous exploration of the giant components of two independent random graphs $\GG^{(i)}(n,p)$, $i=1,2$.
		The difference lies in the $O(n^{1/3})$ starting vertices, but \Thm~\ref{thm:2} holds whenever $np \to c$.
		Therefore, these extra vertices do not change the limiting distribution (as we only care about precision of order $o(\sqrt{n})$).
		As a consequence, for some random variables $\eta_n'$, $\eta_n''$ of order $o_{\Pr}(\sqrt{n})$, the following holds:
		\[
		(\tau_n + \eta_n', \; Z_n + \eta_n'') \overset{d}{=} \left(\min(C^{(1)}(n), C^{(2)}(n)), \; Y_n\right).
		\]
	\end{proof}

	The proof of \Prop~\ref{prop:5} completes the proof of \Thm~\ref{thm:4}.


\begin{thebibliography}{99}
		\bibitem{ErdosRenyi} P.~Erd\H{o}s, A.~R\'enyi: On the Evolution of Random Graphs. Publication of the Mathematical Institute of the Hungarian Academy of Sciences {\bf 5} (1960) 17--61.

		\bibitem{Stepanov} V.E.~Stepanov: On the Probability of Connectedness of a Random Graph $\mathcal{G}_m(t)$. Theory Probab. Appl. {\bf 15} (1970) 55--67.

		\bibitem{BollobasRiordanWalk} B.~Bollob\'as, O.~Riordan: Asymptotic normality of the size of the giant component via a random walk. Journal of Combinatorial Theory, Series B {\bf 102} (2012) 53--61.

		\bibitem{Aldous} D.~Aldous: Brownian excursions, critical random graphs and the multiplicative coalescent. The Annals of Probability {\bf 25} (1997) 812--854.

		\bibitem{SchmidtPruzanShamir} J.~Schmidt-Pruzan, E.~Shamir: Component structures in the evolution of random hypergraphs. Combinatorica {\bf 5} (1985) 81--94.

		\bibitem{FriezeKrivelevichMartin} A.~Frieze, M.~Krivelevich, R.~Martin: The emergence of a giant component in random subgraphs of pseudo-random graphs. Random Structures and Algorithms {\bf 24} (2004) 42--50.

		\bibitem{AjtaiKomlosSzemeredi} M.~Ajtai, J.~Koml\'os, E.~Szemer\'edi: Largest random component of a $k$-cube. Combinatorica {\bf 2} (1982) 1--7.

		\bibitem{BehrischCojaOghlanKang} M.~Behrisch, A.~Coja-Oghlan, M.~Kang: The order of the giant component of random hypergraphs. Random Structures and Algorithms {\bf 36} (2010) 149--184.

		\bibitem{BollobasRiordanHypergraph} B.~Bollob\'as, O.~Riordan: Asymptotic Normality of the Size of the Giant Component in a Random Hypergraph. Random Structures and Algorithms {\bf 41} (2012) 441--450.

		\bibitem{Johansson} T.~Johansson: The giant component of the random bipartite graph. Master thesis in Engineering and Computational Science (2012).

		\bibitem{DoErdeKang} T.A.~Do, J.~Erde, M.~Kang, M.~Missethan: Component behaviour and excess of random bipartite graphs near the critical point. Electron. J. Comb. {\bf 30 (3)} (2023).

		\bibitem{Wang} M.~Wang: Large random intersection graphs inside the critical window and triangle counts. Electron. J. Probab. {\bf 30} (2025) 1--63.

		\bibitem{Clancy} D.~Clancy Jr.: A central limit theorem for the giant in a stochastic block model. arXiv:2501.01351 (2025).

		\bibitem{ZakharovShabanov} P.~Zakharov, D.~Shabanov: Asymptotic normality of the giant component size in a random bipartite graph. Trudy MIPT {\bf 15} (2023) 23--32 (in Russian).

	\end{thebibliography}
\end{document}